\documentclass[
final
]{dmtcs-episciences}

\usepackage[utf8]{inputenc}
\usepackage{subfigure}
\usepackage{amsmath,amsthm,amssymb,amsfonts,latexsym,graphicx,multicol}
\usepackage{color,textpos}
\usepackage{setspace,tikz,scalefnt,multirow}
\usepackage{lmodern,enumerate}
\usepackage[normalem]{ulem}
\usepackage{url}
\usepackage[numbers,sort&compress]{natbib}
\usepackage{hyperref}
\newtheorem{theorem}{Theorem}[section]
\newtheorem{lemma}[theorem]{Lemma}
\newtheorem{proposition}[theorem]{Proposition}
\newtheorem{corollary}[theorem]{Corollary}
\newcommand{\lif}{\operatorname{lif}}
\newcommand{\oper}{{\boldsymbol\Omega}}
\newcommand{\delop}{{\boldsymbol\Delta}}

\newcommand{\ssA}{\textsf{A}}
\newcommand{\ssB}{\textsf{B}}
\newcommand{\ssC}{\textsf{C}}
\newcommand{\ssD}{\textsf{D}}
\title[Pattern avoidance in forests of binary shrubs]{Pattern avoidance in forests of binary shrubs}

\author{David~Bevan\affiliationmark{1} \and
Derek~Levin\affiliationmark{2}\thanks{Student Blugold Commitment Differential Tuition funds through the University of Wisconsin-Eau Claire Summer Research Experiences for Undergraduates } \and
Peter~Nugent\affiliationmark{2}\thanks{Student Blugold Commitment Differential Tuition funds through the University of Wisconsin-Eau Claire Summer Research Experiences for Undergraduates } \and
Jay~Pantone \affiliationmark{3}\and
Lara~Pudwell \affiliationmark{4}\and
Manda~Riehl\affiliationmark{2}\thanks{University of Wisconsin - Eau Claire Office of Research and Sponsored Programs} \and
ML~Tlachac\affiliationmark{2}\thanks{Student Blugold Commitment Differential Tuition funds through the University of Wisconsin-Eau Claire Summer Research Experiences for Undergraduates }}
\affiliation{
Dept. of Mathematics and Statistics, Open Univ., UK\\
Dept. of Mathematics, Univ. of Wisconsin -- Eau Claire,  USA\\
Dept. of Mathematics, Dartmouth College, USA\\
Dept. of Mathematics and Statistics, Valparaiso Univ., USA}


\received{2015-10-31}

\revised{2016-6-8}

\accepted{2016-7-4}

\begin{document}

\publicationdetails{18}{2016}{2}{8}{1322}
\maketitle

\begin{abstract}
%\addtext{
We investigate pattern avoidance in permutations satisfying some additional restrictions. These are naturally considered in terms of avoiding patterns in linear extensions of certain forest-like partially ordered sets, which we call binary shrub forests.
In this context, we enumerate forests avoiding patterns of length three.
In four of the five non-equivalent cases, we present explicit enumerations by exhibiting bijections with certain lattice paths bounded above by the line $y=\ell x$, for some $\ell\in\mathbb{Q}^+$, one of these being the celebrated Duchon's club paths with $\ell=2/3$.
In the remaining case, we use the machinery of analytic combinatorics to determine the minimal polynomial of its generating function, and deduce its growth rate.
%} %\addtext
\end{abstract}

\section{Introduction}

In this paper, we extend pattern avoidance to a previously unexamined combinatorial structure.  Let $\mathcal{S}_n$ be the set of permutations of length $n$.  First, given permutations $\pi=\pi_1\pi_2 \cdots \pi_n \in \mathcal{S}_n$ and $\rho = \rho_1 \rho_2 \cdots \rho_m \in \mathcal{S}_m$ we say that $\pi$ \emph{contains} $\rho$ as a (classical) pattern if there exist $1 \leq i_1 < i_2 < \cdots < i_m \leq n$ such that $\pi_{i_a} < \pi_{i_b}$ if and only if $\rho_a < \rho_b$.  In this case we say that $\pi_{i_1}\pi_{i_2}\cdots \pi_{i_m}$ is \emph{order-isomorphic} to $\rho$ (denoted $\pi_{i_1}\pi_{i_2}\cdots \pi_{i_m} \sim \rho$) and that $\pi_{i_1}\pi_{i_2}\cdots \pi_{i_m}$ \emph{reduces to} $\rho$.  If $\pi$ does not contain $\rho$, then $\pi$ is said to \emph{avoid} $\rho$. This definition of pattern avoidance in permutations appears in many differing applications ranging from the analysis of sorting algorithms to algebraic geometry, and has generated a number of enumeration and classification questions that are of interest in their own right.

Motivated by work with trees~\cite{VERUM10, VERUM11, Rowland10} and comb posets~\cite{Combs14}, Levin, Pudwell, Riehl and Sandberg~\cite{ourpaper} considered pattern avoidance in heaps.  In particular a \emph{complete $k$-ary tree} is a tree where each node has $k$ or fewer children, all levels except possibly the last are completely \emph{full} (i.e. level $i$ contains $k^{i-1}$ vertices), and the last level has all of its nodes to the left side (i.e. for any two vertices in the penultimate level, if the right vertex has a positive outdegree, then the outdegree of the left vertex is $k$, and no more than one vertex in the penultimate level has outdegree not equal to $0$ or~$k$).

A \emph{$k$-ary heap} is a complete $k$-ary tree labeled with $\{1,\dots,n\}$ such that every child has a larger label than its parent.  We draw trees (respectively heaps) with the root at the bottom of the figure.   An example of a 2-ary (i.e. binary) heap on 9 vertices is shown in Figure~\ref{sampleheap}.  Let $\mathcal{H}^k_n$ denote the set of $k$-ary $n$-vertex heaps.  The heap in Figure~\ref{sampleheap} is a member of $\mathcal{H}^2_{9}$.  Given a heap $H$, we associate a permutation $\pi_H$ with it by recording the vertex labels as they are encountered in a breadth-first search.  For example, if $H$ is the heap in Figure~\ref{sampleheap}, then $\pi_H = 125349867$.  We say that heap $H$ contains (respectively avoids) $\rho$ as a pattern if $\pi_H$ contains (respectively avoids) $\rho$ as a classical pattern, using the definition above. Let $\mathcal{H}^k_n(P)$ be the set of members of $\mathcal{H}^k_n$ that avoid all patterns in the list $P$.  While the heap in Figure~\ref{sampleheap} contains 123, 132, 213, 312, and 321, it is a member of $\mathcal{H}^2_{9}(231)$.  In~\cite{ourpaper}, the authors determined $\left|\mathcal{H}^k_n(\rho)\right|$ for $\rho \in \left(\mathcal{S}_3 \setminus \{321\}\right)$, and $\left|\mathcal{H}^k_n(P)\right|$ when $\left|P\right| \geq 2$ and $P \subseteq \mathcal{S}_3$.

\begin{figure}
\begin{center}
$H =$\raisebox{-0.5\height}{
\begin{tikzpicture}[
  level distance=10mm,
  every node/.style={circle,inner sep=.35pt,fill=black},
  level 1/.style={sibling distance=15mm},
  level 2/.style={sibling distance=10mm},
  level 3/.style={sibling distance=5mm}
]
\node[label=below:1] {\phantom{x}} [grow=up]
  child {node[label=right:5] {\phantom{x}}
	  child {node[label=right:8] {\phantom{x}}}
    child {node[label=right:9] {\phantom{x}}}
  }
	child {node[label=left:2] {\phantom{x}}
    child {node[label=left:4] {\phantom{x}}}
    child {node[label=left:3] {\phantom{x}}
		 	  child {node[label=above:7] {\phantom{x}}}
        child {node[label=above:6] {\phantom{x}}}
		}
  };
\end{tikzpicture}}
\end{center}
\caption{A binary heap on 9 vertices}
\label{sampleheap}
\end{figure}
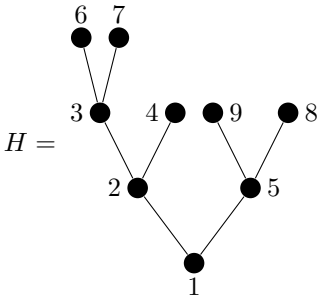

In this paper we extend their pattern avoidance in a new direction by considering forests of heaps.    A \emph{heap forest} is an ordered collection of heaps.  Given a forest $F$ of heaps $H_1,H_2, \dots, H_n$, we label all vertices in $F$ with distinct integers from $\{1,\dots,\left|F\right|\}$ (where $|F| = |H_1| + \cdots + |H_n|$), and then associate the permutation $\pi_{F} = \pi_{H_1}\pi_{H_2}\cdots \pi_{H_n}$.  In other words, we concatenate the associated permutations for each heap to obtain the permutation associated to the forest.  Given the forest in Figure~\ref{sampleforest}, $\pi_{F} = 165(10)92438(11)7(12)(13)$.  As before, we say that forest $F$ avoids pattern $\rho$ if $\pi_{F}$ avoids $\rho$.  Note, from our example, that forests can be composed of heaps with varying numbers of vertices or even heaps that are $k$-ary for different values of $k$.

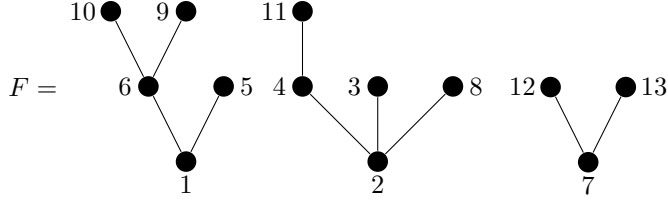
\begin{figure}
\begin{center}

$F =$\raisebox{-0.5\height}{
\begin{tikzpicture}[
  level distance=10mm,
  every node/.style={circle,inner sep=.35pt,fill=black},
  level 1/.style={sibling distance=10mm},
  level 2/.style={sibling distance=10mm},
  level 3/.style={sibling distance=5mm}
]
\node[label=below:1] {\phantom{x}} [grow=up]
  child {node[label=right:5] {\phantom{x}}}
	child {node[label=left:6] {\phantom{x}}
    child {node[label=left:9] {\phantom{x}}}
    child {node[label=left:10] {\phantom{x}}}
  };
\end{tikzpicture}}\hspace{-0.1in}
\raisebox{-0.5\height}{
\begin{tikzpicture}[
  level distance=10mm,
  every node/.style={circle,inner sep=.35pt,fill=black},
  level 1/.style={sibling distance=10mm},
  level 2/.style={sibling distance=10mm},
  level 3/.style={sibling distance=5mm}
]
\node[label=below:{2}] {\phantom{x}} [grow=up]
  child {node[label=right:{8}] {\phantom{x}}}
	child {node[label=left:{3}] {\phantom{x}}}
       child {node[label=left:{4}] {\phantom{x}}
       child {node[label=left:{11}] {\phantom{x}}}}
  ;
\end{tikzpicture}}
\raisebox{-0.8\height}{
\begin{tikzpicture}[
  level distance=10mm,
  every node/.style={circle,inner sep=.35pt,fill=black},
  level 1/.style={sibling distance=10mm},
  level 2/.style={sibling distance=10mm}
]
\node[label=below:7] {\phantom{x}} [grow=up]
  child {node[label=right:13] {\phantom{x}}}
	child {node[label=left:12] {\phantom{x}}
  };
\end{tikzpicture}}
\end{center}

\caption{A 13-node binary heap forest}
\label{sampleforest}
\end{figure}

The consideration of heap forests introduces a number of new parameters to our problem, so we restrict our work to forests of binary (or more generally $k$-ary) shrubs.  A \emph{shrub} is a tree whose root has only leaves as children.  In a binary shrub, each root vertex has exactly two descendants, so a binary shrub forest has $3n$ vertices where $n$ is the number of heaps in the forest, while similarly in a $k$-ary shrub forest we have $(k+1)n$ vertices.

%\note{Corrected to define notation before use and to define the notation actually needed.}
Let $\mathcal{F}^k_n$ be the set of all $k$-ary shrub forests of $n$ heaps.  In Figure~\ref{sampleshrub}, we see a member $F$ of $\mathcal{F}_{4}^2$ where $\pi_F = (10)(12)(11)129348576$.  Let $\mathcal{F}^k_n(P)$ be the set of members of $\mathcal{F}^k_n$ that avoid all patterns in list $P$ and $$\mathcal{S}_n^2(P) = \left\{\pi \in \mathcal{S}_{3n} \mid \pi = \pi_f \text{ for some } f \in \mathcal{F}_n^2(P)\right\}.$$
Equivalently, $$\mathcal{S}_n^2(P) = \left\{\pi \in \mathcal{S}_{3n}(P) \mid  \pi_{3i+1}<\pi_{3i+2} \text{ and }\pi_{3i+1}<\pi_{3i+3} \text{  for all  }0\leq i<n \right\}. $$

Our main goal is to determine $\left|\mathcal{S}_n^2(P)\right|$ where $P \subseteq \mathcal{S}_3$.

\begin{figure}
\begin{center}

$F =$\raisebox{-0.5\height}{
\begin{tikzpicture}[
  level distance=10mm,
  every node/.style={circle,inner sep=.35pt,fill=black},
  level 1/.style={sibling distance=10mm},
  level 2/.style={sibling distance=10mm},
  level 3/.style={sibling distance=5mm}
]
\node[label=below:10] {\phantom{x}} [grow=up]
  child {node[label=right:11] {\phantom{x}}}
	child {node[label=left:12] {\phantom{x}}};
\end{tikzpicture}}\hspace{-0.1in}
\raisebox{-0.5\height}{
\begin{tikzpicture}[
  level distance=10mm,
  every node/.style={circle,inner sep=.35pt,fill=black},
  level 1/.style={sibling distance=10mm},
  level 2/.style={sibling distance=10mm},
  level 3/.style={sibling distance=5mm}
]
\node[label=below:{1}] {\phantom{x}} [grow=up]
  child {node[label=right:{9}] {\phantom{x}}}
	child {node[label=left:{2}] {\phantom{x}}};
\end{tikzpicture}}\hspace{-0.1in}
\raisebox{-0.5\height}{
\begin{tikzpicture}[
  level distance=10mm,
  every node/.style={circle,inner sep=.35pt,fill=black},
  level 1/.style={sibling distance=10mm},
  level 2/.style={sibling distance=10mm},
  level 3/.style={sibling distance=5mm}
]
\node[label=below:3] {\phantom{x}} [grow=up]
  child {node[label=right:8] {\phantom{x}}}
	child {node[label=left:4] {\phantom{x}}};
\end{tikzpicture}}\hspace{-0.1in}
\raisebox{-0.5\height}{
\begin{tikzpicture}[
  level distance=10mm,
  every node/.style={circle,inner sep=.35pt,fill=black},
  level 1/.style={sibling distance=10mm},
  level 2/.style={sibling distance=10mm},
  level 3/.style={sibling distance=5mm}
]
\node[label=below:{5}] {\phantom{x}} [grow=up]
  child {node[label=right:{6}] {\phantom{x}}}
	child {node[label=left:{7}] {\phantom{x}}};
\end{tikzpicture}}

\end{center}
\caption{A binary shrub forest with 12 vertices}
\label{sampleshrub}
\end{figure}
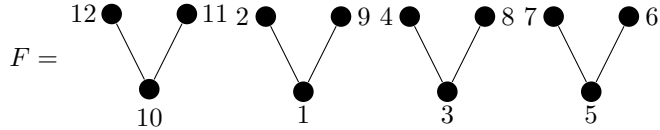

%\delgray{
%\note{I think this paragraph and Table~\ref{bounds} should be deleted.}
%Before determining $\left|\mathcal{S}_n^2(P)\right|$ for specific choices of $P$, we give two naive upper bounds on $\left|\mathcal{S}_n^2(P)\right|$.  First, notice that $\left|\mathcal{F}_n^2\right| = {(3n)!}/{3^n}$ since of the $(3n)!$ permutations of $[3n]=\{1,\dots,3n\}$, only the permutations where $\pi_{3i+1}<\pi_{3i+2}$ and $\pi_{3i+1}<\pi_{3i+3}$ for $0 \leq i < n$, are valid labelings of a binary shrub forest.  A more restrictive bound occurs when $P$ contains a pattern from $\mathcal{S}_3$.  We know that for $\rho \in \mathcal{S}_3$, $\left|\mathcal{S}_n(\rho)\right| = C_n$, where $C_n = \frac{1}{n+1}{\binom{2n}{n}}$ is the $n$th Catalan number.   Since $\mathcal{S}_n^2(\rho) \subseteq \mathcal{S}_{3n}(\rho)$, when $\rho \in \mathcal{S}_3$, we have that $\left|\mathcal{S}_n^2(\rho)\right| \leq C_{3n}$.  Direct computation shows that although ${(3n)!}/{3^n}$ is the tighter bound for $n\leq 2$, for $n\geq 3$, $C_{3n}$ is much more restrictive. See Table~\ref{bounds}.
%} % \delgray

%\delgray{
%\begin{table}
%\delgray{
%\begin{center}
%\begin{tabular}{|c|cccc|c|}
%\hline
%&\multicolumn{4}{|c|}{Terms}&OEIS\\
%\hline
%$C_{3n}$&5,& 132,& 4862,& $208012,\dots$&A187357\\
%\hline
%$\frac{(3n)!}{3^n}$&2,&80,& 13440,& $5913600,\dots$&A210277\\
%\hline
%\end{tabular}
%\end{center}
%\caption{Two naive upper bounds on $\left|\mathcal{S}_{n}^2(\rho)\right|$ for $\rho \in \mathcal{S}_3$}
%\label{bounds}
%} % \delgray
%\end{table}
%} % \delgray

In the rest of this paper, we determine $\left|\mathcal{S}_n^2(\rho)\right|$ exactly for each $\rho\neq 321$.
 A list of sequences and corresponding reference numbers from the On-Line Encyclopedia of Integer Sequences~\cite{OEIS} is given in Table~\ref{results}. This includes results for $\mathcal{S}_n^2(P)$ where $P$ contains more than one pattern of length 3.  Details of the enumerations when $\left|P\right| >1$ are omitted from this paper due to length, but can be found as an ancillary file attached to this arXiv submission at
 \begin{center}
 	\href{http://arxiv.org/src/1510.08036/anc/MultiplePatterns.pdf}{http://arxiv.org/src/1510.08036/anc/MultiplePatterns.pdf}.
 \end{center}
 %at {\color{red}{(ADD ARXIV reference)}}.

%\addtext{
%\note{Added summary of results / overview.}
In the next section, we look at each of the sets avoiding a single pattern of length 3.
For five of the patterns, we enumerate the set by establishing a bijection with a family of lattice paths from $(0,0)$ to $(m,\ell m)$ for some $\ell\in\mathbb{Q}^+$, bounded above by the line $y=\ell x$.
In four cases, these lattice paths consist of unit east and north steps.
$\mathcal{S}_n^2(123)$ is shown to be equinumerous to such paths with $\ell=2$ (Theorem~\ref{T:av123}), and more generally $\mathcal{F}^k_n(123)$ is equinumerous to such paths with $\ell=k$ (Theorem~\ref{T:av123f}).
$\mathcal{S}_n^2(132)$ is in bijection with these paths with $\ell=3$. Indeed, we prove a more general result concerning $k$-ary forests, that $\mathcal{F}^k_n(132)$ is equinumerous
to these lattice paths with $\ell=k+1$ (Theorem~\ref{T:av132}).
In the case of $\mathcal{S}_n^2(231)$, we establish a bijection with the celebrated case of paths bounded above by $y=\frac{2}{3}x$, the so-called ``Duchon's club model'' (Theorem~\ref{T:av231}), and more generally we outline how $\mathcal{F}^k_n(231)$ is equinumerous to such paths bounded above by $y=\frac{k}{k+1}x$.

We prove that
$\mathcal{S}^2_n(213)$ and $\mathcal{S}^2_n(312)$ are equinumerous (Theorem~\ref{T:av213eq312}) and establish a bijection with lattice paths having east, north and northeast steps, bounded by $y=3x$ (Theorem~\ref{T:av213paths}).
Finally, we investigate $\mathcal{S}_n^2(321)$. We are unable to enumerate this set explicitly. However, using functional equations,
we are able to generate nearly a thousand terms of its enumeration sequence,
prove that its
generating function is algebraic and
determine its minimal polynomial and growth rate (Theorem~\ref{thmMinPoly}).
We conclude in Section~\ref{sectSummary} with some questions.
%} %\addtext

\begin{table}[!hbt]
\begin{center}
\begin{tabular}{|l|l|r|l|}

\hline
$P$& $\left(\left|\mathcal{S}_n^2(P)\right|\right)_{n \geq 1}$ & OEIS~~~~~&Result\\
\hline
$\emptyset$ & $2,80,13440,5913600,\ldots$ & A210277& \\
\hline
123 & $1,3,12,55,273,\ldots$ & A001764& Theorem~\ref{T:av123}\\
\hline
132 & $1,4,22,140,969,\ldots$ &  A002293& Corollary~\ref{C:av132}\\
\hline
213 & \multirow{2}{*}{$2,14,134,1482,17818,\ldots$} & \multirow{2}{*}{A144097}&\multirow{2}{*}{Theorem~\ref{T:av213and312}}\\
312&&&\\
\hline
231 & $2,23,377,7229,151491,\ldots$ & A060941& Theorem~\ref{T:av231}\\
\hline
321 & $2,37,866,23285,679606,\ldots$ & A257995 & Theorem~\ref{thmFunctEqn}\\
\hline
132,213& \multirow{2}{*}{$1,2,4,8,16,\ldots$} & \multirow{2}{*}{A000079}&\multirow{2}{*}{}\\
132,312&&&\\
\hline
132,321 & $1,4,10,19,31,\ldots$ & A005448&\\
\hline
213,231 & \multirow{2}{*}{$2,8,32,128,512,\ldots$} & \multirow{2}{*}{A004171}&\multirow{2}{*}{}\\
231,312&&&\\
\hline
213,312 & \multirow{4}{*}{$2,2,2,2,2,\ldots$} & \multirow{4}{*}{A007395}&\multirow{4}{*}{}\\
213,231,312&&&\\
213,312,321&&&\\
213,231,312,321&&&\\
\hline
213,321 & $2,6,13,23,36,\ldots$ & A143689&\\
\hline
231,321 & $2,12,72,432,2592,\ldots$ & A167747& \\
\hline
312,321 & $2,10,50,250,1250,\ldots$ & A020699&\\
\hline
132,213,321 & $1,2,3,4,5,\ldots$ & A000027& \\
\hline
213,231,321 & $2,4,6,8,10,\ldots$ & $\text{A005843}$& \\
\hline
231,312,321 & $2,6,18,54,162,\ldots$ & A025192&\\
\hline

\end{tabular}
\end{center}
\caption{$\left|\mathcal{S}_n^2(P)\right|$ where $P \subseteq \mathcal{S}_3$}
\label{results}
\end{table}

\section{Avoiding a pattern of length $3$}\label{sect1patt}

%\addtext{
To enumerate our first two sets, we make use of the following result, concerning lattice paths in a wedge, first proved by
Fuss at the end of the 18th century. See~\cite[Section~12.1]{Loehr2011} for a modern presentation of the proof.

\begin{proposition}[{Fuss~\cite{Fuss1791}}]\label{propFuss}%\note{Added a new proposition, used in proofs of Theorems~\ref{T:av123} and~\ref{T:av132}.}
The number of lattice paths with unit East and North steps from $(0,0)$ to $(m,\ell m)$ remaining weakly under the line $y=\ell x$ is given by
$$\frac{1}{\ell m+1}\binom{(\ell+1)m}{m}.$$
\end{proposition}
%} %\addtext

\subsection{Avoiding 123}

%\addtext{
We enumerate forests avoiding 123 by exhibiting a bijection with lattice paths bounded above by $y=kx$.
We begin by establishing the fact that the labels of a 123-avoiding forest are uniquely determined by their root labels.
%} % \addtext

\begin{lemma}\label{L:DecLabels}For any $F \in \mathcal{F}^k_n(123)$ with $k \geq 2$ and $n\geq 0$, the roots of the forest of $F$ have decreasing labels as do the leaves.  \end{lemma}
\begin{proof}
Consider $\pi_F$ and suppose to the contrary that $\pi_i$ and $\pi_j$ are root labels where $i<j$ and $\pi_i<\pi_j$.  Then $\pi_i\pi_j\pi_k \sim 123$ where $\pi_k$ is a descendant of $\pi_j$.  Similarly if $\pi_i$ and $\pi_j$ are leaf labels where $i<j$ and $\pi_i<\pi_j$, then $\pi_k\pi_i\pi_j \sim 123$ where $\pi_k$ is the parent of $\pi_i$. \end{proof}

 Therefore if we know only the root labels of a labeled forest avoiding $123$, the labels for the entire forest are uniquely determined.  Note however, this is not the same as saying that any labeling of roots corresponds with a forest that avoids 123.

%\addtext{
We are now in a position to enumerate $\mathcal{F}_n^k(123)$ and consequently $\mathcal{S}^2_n(123)$.
%}

\begin{theorem}\label{T:av123f}
For $k \geq 2$ and $n \geq 0$,
$$\left|\mathcal{F}_n^k(123)\right|=\frac{1}{k n+1}\binom{(k+1)n}{n}.$$
\end{theorem}

\begin{proof}
We will show that $\mathcal{F}^k_n(123)$ is in bijection with NE lattice paths from $(0,0)$ to $(n,kn)$ weakly below the line $y=kx$.
The result then follows by Proposition~\ref{propFuss} with $\ell=k$.%\note{Added use of Fuss' Proposition.}

A 123-avoiding forest of $k$-ary shrubs is uniquely determined by its roots.  Therefore, we must choose a set of $n$ root labels, $1=r_1<r_2<\cdots <r_n$ such that there are at least $k$ unused labels larger than $r_n$, at least $2k$ unused labels larger than $r_{n-1}$, and in general at least $k(n-i+1)$ unused labels larger than $r_i$ for $1 \leq i \leq n$. Then the roots of the forest have labels $r_n, r_{n-1}, \dots, r_2, r_1$, the leaves use the remaining labels in decreasing order, and each tree has leaves with larger labels than the roots.

Similarly, in a NE lattice path from $(0,0)$ to $(n,kn)$ weakly below the line $y=kx$, we must choose $n$ steps to be East steps.  To stay below the line $y=kx$, the $i$th east step must have $k(n-i+1)$ north steps after it ($1 \leq i \leq n$).

Thus, 123-avoiding $k$-ary shrub forests are in bijection with NE lattice paths below the line $y=kx$ in the following way: given $F \in \mathcal{F}^k_n(123)$, let $r_i=\pi_{(k+1)(i-1)+1}$ for $1 \leq i \leq n$ be the roots of the forest corresponding to $\pi_F$.  Now construct a NE-lattice path of $n$ East steps and $kn$ North steps so that the $j$th step is an East step if and only if $j \in \{r_1, \dots, r_n\}$.  This map is easily invertible.  Given a NE-lattice path below $y=kx$, let $r_1< \cdots <r_n$ be the positions of the East steps.  Then use $r_1, \dots, r_n$ in decreasing order as the labels of the roots for a forest, and place the unused labels in decreasing order on the roots.
\end{proof}

We then obtain our desired result by setting $k=2$.

\begin{corollary}\label{T:av123}
$\left|\mathcal{S}^2_n(123)\right| = \dfrac{1}{2n+1}{\dbinom{3n}{n}}$.
\end{corollary}

The 3 members of $\mathcal{S}^2_2(123)$ and their corresponding NE lattice paths are shown in Figure~\ref{av123}.

\begin{figure}
\begin{center}
\raisebox{-0.3\height}{
\begin{tabular}{c}\scalebox{0.4}{\begin{tikzpicture}
\draw[help lines] (0,0) grid (2,4);
\draw[line width=2pt] (0,0) --(2,0)--(2,4);
\draw[dashed] (0,0)--(2,4);
\end{tikzpicture}}\\
\textbf{EE}NNNN\\ \textbf{2}65\textbf{1}43\end{tabular}}
\raisebox{-0.3\height}{
\begin{tabular}{c}\scalebox{0.4}{\begin{tikzpicture}
\draw[help lines] (0,0) grid (2,4);
\draw[line width=2pt] (0,0) --(1,0) --(1,1)--(2,1)--(2,3)--(2,4);
\draw[dashed] (0,0)--(2,4);
\end{tikzpicture}}\\
\textbf{E}N\textbf{E}NNN\\ \textbf{3}65\textbf{1}42\end{tabular}}
\raisebox{-0.3\height}{
\begin{tabular}{c}\scalebox{0.4}{\begin{tikzpicture}
\draw[help lines] (0,0) grid (2,4);
\draw[line width=2pt] (0,0) --(1,0) --(1,2)--(2,2)--(2,4);
\draw[dashed] (0,0)--(2,4);
\end{tikzpicture}}\\
\textbf{E}NN\textbf{E}NN\\ \textbf{4}65\textbf{1}32\end{tabular}}
\end{center}
\caption{$\mathcal{S}^2_2(123)$ and corresponding lattice paths}
\label{av123}
\end{figure}
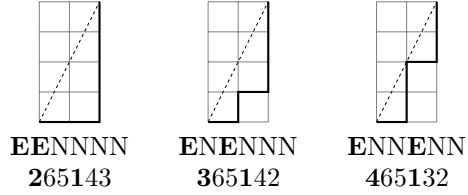
\vspace{.1in}

%The proof of Theorem~\ref{T:av123f} is identical to the proof of Theorem~\ref{T:av123}, replacing all instances of 2 with $k$, and $r_i =\pi_{3(i-1)+1}$ with $r_i=\pi_{(k+1)(i-1)+1}$.

\subsection{Avoiding 132}

%\addtext{
We enumerate forests avoiding 132 by giving a bijection with lattice paths bounded above by $y=3x$.
Indeed, we establish a stronger result, applicable to all $k$-ary shrub forests, exhibiting a bijection between $\mathcal{F}^k_n(132)$ and lattice paths bounded in the wedge below $y=(k+1)x$.

As with forests avoiding 123, we start by establishing that the labels of a 132-avoiding forest are also determined by their root labels.
%} % \addtext

\begin{lemma}\label{L:av132}
The labels of a 132-avoiding k-ary shrub forest are uniquely determined by their root labels.
\end{lemma}

\begin{proof}
Define $S$ to be the ordered set of currently unused labels in a forest of $t$ heaps with $v=k+1$ vertices each, and imagine we are assigning labels from left to right on heaps, and in breadth-first order on each heap. Initially, $S = \{1, 2, \ldots, vt\}$. The labels in each heap are clearly larger than the label of their root, and in each heap are increasing (in order to avoid $132$). In fact, not only are they larger, they are the smallest possible unused labels. In other words, the leaves of each root are labelled, from left to right, by the $k$ smallest unused labels greater than the root. If not, the next largest label $y$ would be used later (on its right) and would create a $xzy \sim 132$ pattern where $z$ is a leaf of (the root) $x$ but not $y$. Thus we only need know the label of the roots in order to deduce the entire labeling of a $132$-avoiding k-ary shrub forest of heaps.  Furthermore, knowing the first root and the longest decreasing subsequence of roots is sufficient since if the roots are $r_1, r_2, \dots, r_t$ and we have $r_i < r_{i+1} < \cdots < r_{j-1} > r_j$ then $r_i > r_j$ in order to avoid 132.
\end{proof}

Again, Lemma~\ref{L:av132} is not equivalent to saying that any labeling of roots corresponds with a forest that avoids 132.

It turns out that we can use the structure of 132-avoiding forests to describe not only $\mathcal{S}^2
_n(132)$, but more generally $\mathcal{F}^k_n(132)$.

\begin{theorem}
\label{T:av132}
$\big|\mathcal{F}^k_n(132)\big| = \dfrac{1}{(k+1)n+1}\dbinom{(k+2)n}{n}$.
\end{theorem}

\begin{proof}We provide a bijection from $\mathcal{F}^k_n(132)$ to the set $P$ of paths under the line $y=(k+1)x$ from $(0,0)$ to $(n,(k+1)n)$ using North $(0,1)$ and East $(1,0)$ steps. %\addtext{
The result then follows by Proposition~\ref{propFuss} with $\ell=k+1$.%}\note{Added use of Fuss' Proposition.}

Define $\phi: P \to \mathcal{F}^k_n(132)$ as follows.  Given $p \in P$, let $0 = w_1 \leq w_2 \leq \cdots \leq w_n$ be the heights of the East steps in $p$.  Let $w^{\prime} := (w_n+1)(w_{n-1}+1)\cdots (w_1+1)$.  Label the first root $w^{\prime}_1$.  For each subsequent root, if $w^{\prime}_i \neq w^{\prime}_{i-1}$, then label root $i$ with $w^{\prime}_i$.  If $w^{\prime}_i = w^{\prime}_{i-1}$, leave root $i$ unlabeled.  Now, by Lemma~\ref{L:av132} label the remaining forest vertices from left to right (and in breadth-first order on each heap), at each point using the smallest unused label larger than the most recent root.  We claim this forest avoids 132.  Suppose there were a 132 pattern not involving a root where $\ell$ plays the role of 1.  Since $\ell$'s root is smaller than $\ell$, there is another copy of 132 using $\ell$'s root.  Now, by construction, all entries larger than a given root and after the root appear in increasing order, so this forest is 132-avoiding.

Next we describe $\phi^{-1}: \mathcal{F}^k_n(132) \to P$, beginning with a 132-avoiding forest. We create a string $y$, with $y_1$ being the root of the first heap. We add an element to the string $y$ for each root $r_i$. If the permutation has been only increasing since $r_{i-1}$, set $y_i = y_{i-1}$. If the permutation has had a descent (which could only occur between the last leaf of the tree with root $r_{i-1}$ and the root $r_i$), set $y_i = r_i$. We then create a string $y'$ by subtracting $1$ from each element of $y$. Finally we reverse the string to obtain $y'^r$, which is our string displaying heights of the East steps in our path, which we claim lies below the line $y=(k+1)x$.  Notice that a path lies below $y=(k+1)x$ if and only if $y'^r_i \leq (i-1)(k+1)$ for $i \geq 1$.  By construction, $y_1=r_1 \leq n-k$ since $r_1$ has $k$ leaves larger than itself.  In general, if $r_i^* = \min(r_1, \dots, r_i)$, then $r_i^* \leq n-((i-1)+ki) = n+1 -(k+1)i$ since there are $i-1$ other roots and $ki$ leaves larger than $r_i$ as labels on the first $i$ shrubs.  Since $y_i = r_i^*$, we have that $y_i \leq  n+1 -(k+1)i$ for all $i$, and thus $y'_i \leq n-(k+1)i$ for all $i$. Let $t=\frac{n}{k+1}$ be the number of trees in our shrub forest.  Our bound on $y'_i$ implies that $y'^r_{t+1-i} \leq n-(k+1)(t+1-i) = t(k+1)-(k+1)(t+1-i) = (i-1)(k+1)$, as desired.
\end{proof}

This proof gives two easy corollaries. The first demonstrates that although we considered shrubs here, in special cases (such as avoiding $132$) we can characterize the shrub condition more simply in terms of a permutation composed of a string of equal-length subpermutations. The second is our result restricted to binary shrubs.

\begin{corollary}
Let $\sigma$ be a permutation composed of a concatenation of $m$ increasing sequences of length $n$. The number of such $\sigma \in \mathcal{S}_{nm}(132)$ is given by $$\frac{1}{nm+1}\binom{(n+1)m}{m}.$$
\end{corollary}

Also, $\left|\mathcal{S}^2_n(132)\right|$ is a special case of Theorem~\ref{T:av132}.

\begin{corollary}\label{C:av132}
$\left|\mathcal{S}^2_n(132)\right| = \left|\mathcal{F}^2_n(132)\right| = \dfrac{1}{3n+1}{\dbinom{4n}{n}}$.
\end{corollary}

Figure~\ref{av132} shows two examples of lattice paths below $y=3x$ that correspond to members of $\mathcal{S}^2(132)$.

\begin{figure}
\begin{center}
\begin{tabular}{ccc}
\raisebox{-0.5\height}{\scalebox{0.4}{
\begin{tikzpicture}
\draw[help lines] (0,0) grid (3,9);
\draw[line width=2pt] (0,0) --(1,0) --(1,3)--(2,3)--(2,4)--(3,4)--(3,9);
\draw[black,dashed] (0,0)--(3,9);
\end{tikzpicture}}}
&$\qquad$&
\raisebox{-0.5\height}{\scalebox{0.4}{
\begin{tikzpicture}
\draw[help lines] (0,0) grid (3,9);
\draw[line width=2pt] (0,0) --(1,0) --(1,2)--(3,2)--(3,9);
\draw[black,dashed] (0,0)--(3,9);
\end{tikzpicture}}}\\
$w=034=y'^{r}$&&$w=022=y'^{r}$\\
$w^{\prime}=541=y$&&$w^{\prime}=331=y$\\
$\pi = \textbf{5}67\phantom{.}\textbf{4}89\phantom{.}\textbf{1}23$&&$\pi=\textbf{3}45\phantom{.}\mathbf{6}78\phantom{.}\textbf{1}29$\\
\end{tabular}
\end{center}
\caption{Two lattice paths and the permutations corresponding to the 132-avoiding binary shrubs}
\label{av132}
\end{figure}
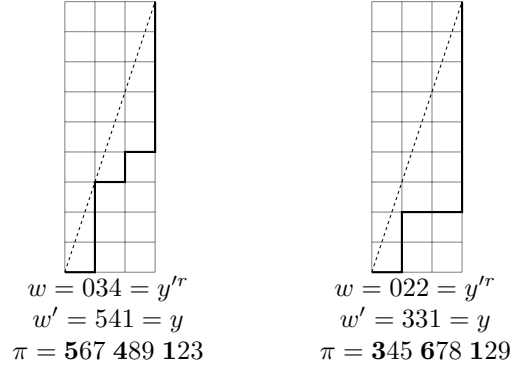

\subsection{Avoiding 213 or 312}

The only pair of permutations $\rho_1,\rho_2 \in \mathcal{S}_3$ for which
$\left|\mathcal{S}^2(\rho_1)\right| = \left|\mathcal{S}^2(\rho_2)\right|$ is $\{\rho_1,\rho_2\}=\{213,312\}$.
%\addtext{
We prove this equivalence in Theorem~\ref{T:av213eq312} below.

To enumerate these sets, we make use of a variant of the result of Fuss we used above, in which diagonal, as well as horizontal and vertical, steps are permitted.

\begin{proposition}[{Schr\"oder~\cite[Theorem~2.9]{Schroder2007}}]\label{propSchroder}%\note{Added new propositions, used in the proof of Theorem~\ref{T:av213and312} to give the exact enumeration.}
The number of lattice paths with steps $(0,1)$, $(1,0)$ and $(1,1)$, from $(0,0)$ to $(m,\ell m)$, remaining weakly below the line $y=\ell x$ is given by
$$\frac{1}{\ell m+1}\sum_{v=0}^{m}\binom{\ell m+1}{m-v}\binom{\ell m+v}{v}.$$
\end{proposition}
We also require the following `folklore' bijection between two families of lattice paths. This proposition follows directly from the invertibility of the affine map of the Euclidean plane implied in its statement.
\begin{proposition}[{Banderier and Wallner~\cite[Proposition~2.1]{BW2015}}]\label{propPathBij}
Lattice paths with step set $S$, from $(0,0)$ to $(m,\ell m)$ remaining weakly below $y=\ell x$
are in bijection with
lattice paths with step set $\{(x+y,\ell x-y):(x,y)\in S\}$, from $(0,0)$ to $((\ell+1)m,0)$ remaining weakly above the $x$-axis.
\end{proposition}

We enumerate forests avoiding 213 by exhibiting a bijection with a family of lattice paths which, by Proposition~\ref{propPathBij}, are known to be equinumerous to paths bounded above by $y=3x$.
%} % \addtext

\begin{theorem}
\label{T:av213paths}
$\left|\mathcal{S}^2_n(213)\right|  = a_n$ where $a_n$ is the number of lattice paths from $(0,0)$ to $(4n,0)$ with unit steps $(1,3)$, $(2,2)$, and $(1,-1)$ staying weakly above the $x$-axis.
\end{theorem}

\begin{proof}
First we give a correspondence from lattice paths to 213-avoiding permutations.

Begin with a lattice path from $(0,0)$ to $(4n,0)$ that contains only $(1,3)$, $(2,2)$, and $(1,-1)$ steps and stays weakly above the $x$-axis.  Partition the path into \emph{segments} that traverse exactly one unit in the vertical direction.  (Each $(2,2)$ step will be partitioned into 2 segments and each $(1,3)$ step will be partitioned into 3 segments.)  We label each segment of a $(1,3)$ or $(2,2)$ step as well as the midpoint of each $(2,2)$ step with a distinct label from $\{1,2, \dots, 3n\}$ in the following way:  Locate the lowest line $y=i$ from which an unlabeled up-segment begins.  Find the rightmost such segment $s$ beginning at $y=i$, and give $s$ the lowest unused label.  Let $j$ be the number of segments and midpoints to the right of $s$.  We now apply the $j$ smallest unused labels to the subpath to the right of $s$ and the remaining (larger) unused labels to the subpath to the left of $s$, and repeat this construction recursively.  In the case that a subpath has no up-segments, it must consist of a single midpoint of a $(2,2)$-step and so it receives the only unused label reserved for the subpath.  Now, read the labels from left to right to obtain a permutation $\pi$.

Notice that $\pi$ can be the labels of a binary shrub forest since the three labels on each increasing $(1,3)$ step or $(2,2)$ step correspond to the labels of a heap.  By construction, the labels on a $(1,3)$ step form a 123 pattern while the labels on a $(2,2)$ step form a 132 pattern.  Further, $\pi$ avoids 213 since for each digit $i$, all labels larger than and left of $i$ are greater than all labels larger than and right of $i$.

Next, we show how to reverse this correspondence, by giving a map from 213-avoiding permutations to lattice paths.  Start with a binary shrub forest $f$. Consider $f$ one heap at a time.  A heap with an increasing pair of leaves corresponds to a $(1,3)$ step and a heap with a decreasing pair of leaves corresponds to a $(2,2)$ step.  The $(1,3)$ and $(2,2)$ steps can be separated by $(1,-1)$ steps. To determine the placement of the $(1,-1)$ steps, we look at the roots of the heaps.

Take $\pi$ and mark any digit that is the first digit of a pair of leaves with decreasing labels.  Now, for each root, count the number of unmarked digits (roots and unmarked leaves) before the root and larger than the root.  This is the number of $(1,-1)$ steps that immediately precede the increasing step corresponding to that heap.  Mark the digits that were just used and repeat. Finally, end the path by adding $(1,-1)$ steps to return to the $x$-axis.
\end{proof}

As an example, consider the path in Figure~\ref{av213}.  Here, $(2,2)$ steps are shown with doubled lines to make them clearly distinct from $(1,3)$ steps.  For the map from $\pi$ to the path, the heaps 7 15 14, 11 13 12, and 2 4 3 have decreasing leaves, so each of them correspond to a $(2,2)$ step.  The heaps 8 9 10 and 1 5 6 have increasing leaves, so each of them correspond to a $(1,3)$ step.  The increasing steps in the path alternate between $(2,2)$ steps and $(1,3)$ steps.  Now, we mark the digits 15, 13, and 4 in $\pi$.  For the root 8, we see that 14 is unmarked and prior to 8, so we put one $(1,-1)$ step between the first two upsteps and mark 14.  For the root 11, we see no numbers larger than 11 and prior to 11 that have not yet been used, so we put zero $(1,-1)$ steps between the second and third upsteps.  For the root 1, we see 12, 11, 10, 9, 8, and 7 larger than 1 and prior to 1 that have not yet been used, so we put six $(1,-1)$ steps between the third and fourth upsteps and mark 12, 11, 10, 9, 8, and 7.  For the root 2, we see 5 and 6 larger than 2 and prior to 2 that have not yet been used, so we put two $(1,-1)$ steps between the fourth and fifth upsteps and mark 5 and 6.  Finally, we must take three $(1,-1)$ steps at the end of the path to return to the $x$-axis.

\begin{figure}[hbt]
\begin{center}
\raisebox{-0.5\height}{\scalebox{0.6}{
\begin{tikzpicture}
\draw[help lines] (0,0) grid (20,6);
%\draw[double,ultra thick] (0,0)--(2,2) node [above left, pos=0.25] {a} node [above left, pos=0.75] {b} node [above left, pos=0.5] {m};
%\draw[ultra thick] (2,2)--(3,1) node [below left, pos=0.5] {b};
%\draw[black,ultra thick] (3,1)--(4,4) node [above left, pos=0.17] {c} node [above left, pos=0.5] {d} node [above left, pos=0.83] {e};
%\draw[double,ultra thick] (4,4)--(6,6) node [above left, pos=0.25] {f} node [above left, pos=0.75] {g} node [above left, pos=0.5] {n};
%\draw[ultra thick] (6,6)--(7,5) node [below left, pos=0.5] {g};
%\draw[ultra thick] (7,5)--(8,4) node [below left, pos=0.5] {f};
%\draw[ultra thick] (8,4)--(9,3) node [below left, pos=0.5] {e};
%\draw[ultra thick] (9,3)--(10,2) node [below left, pos=0.5] {d};
%\draw[ultra thick] (10,2)--(11,1) node [below left, pos=0.5] {c};
%\draw[ultra thick] (11,1)--(12,0) node [below left, pos=0.5] {a};
%\draw[black,ultra thick] (12,0)--(13,3) node [above left, pos=0.17] {h} node [above left, pos=0.5] {i} node [above left, pos=0.83] {j};
%\draw[ultra thick] (13,3)--(14,2) node [below left, pos=0.5] {j};
%\draw[ultra thick] (14,2)--(15,1) node [below left, pos=0.5] {i};
%\draw[double,ultra thick] (15,1)--(17,3) node [above left, pos=0.25] {k} node [above left, pos=0.75] {l}  node [above left, pos=0.5] {o};
%\draw[ultra thick] (17,3)--(18,2) node [below left, pos=0.5] {l};
%\draw[ultra thick] (18,2)--(19,1) node [below left, pos=0.5] {k};
%\draw[ultra thick] (19,1)--(20,0) node [below left, pos=0.5] {h};
\draw[double,ultra thick] (0,0)--(2,2) node [above left, pos=0.25] {7} node [above left, pos=0.75] {14} node [above left, pos=0.5] {15};
\draw[ultra thick] (2,2)--(3,1) node [below left, pos=0.5] {};
\draw[black,ultra thick] (3,1)--(4,4) node [above left, pos=0.17] {8} node [above left, pos=0.5] {9} node [above left, pos=0.83] {10};
\draw[double,ultra thick] (4,4)--(6,6) node [above left, pos=0.25] {11} node [above left, pos=0.75] {12} node [above left, pos=0.5] {13};
\draw[ultra thick] (6,6)--(7,5) node [below left, pos=0.5] {};
\draw[ultra thick] (7,5)--(8,4) node [below left, pos=0.5] {};
\draw[ultra thick] (8,4)--(9,3) node [below left, pos=0.5] {};
\draw[ultra thick] (9,3)--(10,2) node [below left, pos=0.5] {};
\draw[ultra thick] (10,2)--(11,1) node [below left, pos=0.5] {};
\draw[ultra thick] (11,1)--(12,0) node [below left, pos=0.5] {};
\draw[black,ultra thick] (12,0)--(13,3) node [above left, pos=0.17] {1} node [above left, pos=0.5] {5} node [above left, pos=0.83] {6};
\draw[ultra thick] (13,3)--(14,2) node [below left, pos=0.5] {};
\draw[ultra thick] (14,2)--(15,1) node [below left, pos=0.5] {};
\draw[double,ultra thick] (15,1)--(17,3) node [above left, pos=0.25] {2} node [above left, pos=0.75] {3}  node [above left, pos=0.5] {4};
\draw[ultra thick] (17,3)--(18,2) node [below left, pos=0.5] {};
\draw[ultra thick] (18,2)--(19,1) node [below left, pos=0.5] {};
\draw[ultra thick] (19,1)--(20,0) node [below left, pos=0.5] {};
\end{tikzpicture}}}\\
%$\alpha = $amb cde fng hij kol\\
%$\beta = $hkl oij acd efg nbm\\
$\pi = $\underline{7 15 14} \underline{8 9 10} \underline{11 13 12} \underline{1 5 6} \underline{2 4 3}

\end{center}
\caption{A lattice path and its corresponding 213-avoiding permutation}
\label{av213}
\end{figure}
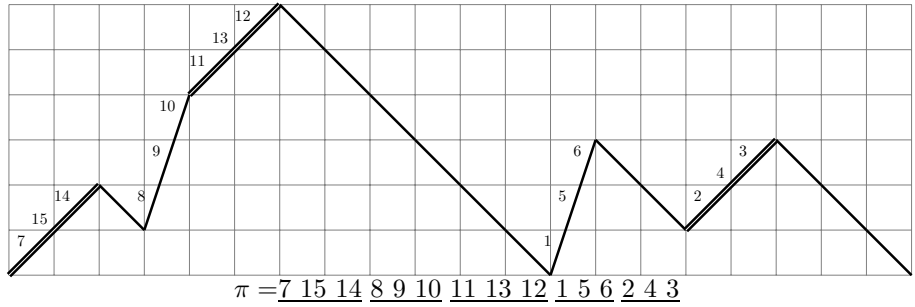

%\addtext{
We now show that 213-avoiding forests are equinumerous to those avoiding 312.%}

\begin{theorem}
\label{T:av213eq312}
$\left|\mathcal{S}^2_n(213)\right|  = \left|\mathcal{S}^2_n(312)\right|$.
\end{theorem}

\begin{proof}
Consider a forest avoiding 213 as it is built from left to right.  According to our correspondence a $(2,2)$ upstep indicates that the next heap forms a 132 pattern while a $(1,3)$ upstep indicates that the next heap in the forest forms a 123 pattern.  The $(1,-1)$ steps between the upsteps indicate the labels of the next in the following way:

%Suppose that the first heap in our forest forms a 123 pattern and we wish to append a single new heap.  All labels that are used on the new heap must be labelled from a set of consecutive numbers.  In particular, this set may contain values which are all larger than the values of the first heap (resulting in the first heap using labels 123), all between the two largest digits (resulting in the first heap using labels 126), all between the two smallest digits (resulting in the first heap using labels 156), or all smaller than the labels of the first heap (resulting in the first heap using labels 456).  In other words, there are 4 choices in how the new labels interleave with previous labels, and accordingly there is a choice of 0, 1, 2, or 3 $(1,-1)$ steps after a $(1,3)$ upstep.

Suppose that the first $k$ upsteps in our path encode a forest with labels $\{1,2,\dots, 3k\}$ where the last heap forms a 123 pattern with labels $\ell-2, \ell-1, \ell$.  We wish to append a single new heap to this forest.  In order for the shrub forest to avoid 213, all labels on the new heap must be consecutive.  In particular, the labels on the new heap may be $i+1, i+2, i+3$ for $0 \leq i \leq \ell$ (and all labels larger than $i$ on the original forest are incremented by 3). Let $d$ be the number of downsteps immediately after the $(1,3)$ upstep corresponding to the last 123 pattern; then $\ell-d=i$ indicates the labels on the next heap to be appended are $i+1$, $i+2$, and $i+3$.

%Similarly, suppose that the first heap in our forest forms a 132 pattern and we wish to append a single new heap.  All labels that are used on the new heap must again be taken from a set of $3$ consecutive numbers, and they may not be larger than the labels of the leaves on our first heap, lest we form a forbidden pattern.  There are 3 choices for what labels to use on the new heap, and accordingly there is a choice of 0, 1, or 2 $(1,-1)$ steps after a $(2,2)$ upstep.

Similarly, suppose that the first $k$ upsteps in our path encode a forest with labels $\{1,2,\dots, 3k\}$ where the last heap forms a 132 pattern with labels $\ell-2, \ell-1, \ell$.  We wish to append a single new heap to this forest.  In order for the shrub forest to avoid 213, all labels on the new heap must be consecutive.  In particular, the labels on the new heap may be $i+1, i+2, i+3$ for $0 \leq i \leq \ell-1$ (and all labels larger than $i$ on the original forest are incremented by 3). Let $d$ be the number of downsteps immediately after the $(2,2)$ upstep corresponding to the last 132 pattern; then $\ell-d-1=i$ indicates the labels on the next heap to be appended are $i+1$, $i+2$, and $i+3$.

In general, when we append a new 123 heap to the end of a 213-avoiding forest, we increase the number of sets of labels we can use on the next heap by 3, and when we append a new 132 heap to the end of a 213-avoiding forest, we increase the number of sets of labels we can use on the next heap by 2.  The number $k$ of $(1,-1)$ downsteps tells us to use the $(k+1)$st highest possible value for the root of the next appended heap.

We can construct 312-avoiding forests in a similar way, encoding 123 heaps with $(1,3)$ steps and 132 heaps with $(2,2)$ steps.

%Suppose that the first heap in our forest forms a 123 pattern and we wish to append a new heap.  All labels that are used on leaves of the new heap must be larger than all existing labels, but the root can appear lower.  In particular, the root may be larger than the values of the first heap (resulting in the first heap using labels 123), between the two largest digits (resulting in the first heap using labels 124), between the two smallest digits (resulting in the first heap using labels 134), or smaller than the labels of the first heap (resulting in the first heap using labels 234).  In other words, there are 4 choices in how the new root interleaves with previous labels, and accordingly there is a choice of 0, 1, 2, or 3 $(1,-1)$ steps after a $(1,3)$ upstep.

Suppose that the first $k$ upsteps in our path encode a forest with labels $\{1,2,\dots, 3k\}$ where the last heap forms a 123 pattern with labels $\ell-2, \ell-1, \ell$.  We wish to append a single new heap to this forest.  In order for the shrub forest to avoid 312, all labels that are used on leaves of the new heap must be larger than all existing labels, but the root can appear lower. In particular, the labels on the new heap may be $i, 3k+2, 3k+3$ for $1 \leq i \leq 3k+1$ (and if $i \leq 3k$ then $i$ may not be the larger digit in an inversion in the original forest on $\{1,2,\dots , 3k\}$). Then, all labels larger than $i-1$ on the original forest are incremented by 1 to obtain the new forest. Let $d$ be the number of downsteps immediately after the $(1,3)$ upstep corresponding to the last 123 pattern; then $\ell-d=i$ indicates the labels on the next heap to be appended are $i$, $3k+2$, and $3k+3$.

%Similarly, suppose that the first heap in our forest forms a 132 pattern and we wish to append a new heap.  All labels that are used on leaves of the new heap must still be larger than all existing labels, but the root may not appear between the largest two digits now. There are 3 choices for what labels to use on the new heap, and accordingly there is a choice of 0, 1, or 2 $(1,-1)$ steps after a $(2,2)$ upstep.

Similarly, suppose that the first $k$ upsteps in our path encode a forest with labels $\{1,2,\dots, 3k\}$ where the last heap forms a 132 pattern with labels $\ell-2, \ell-1, \ell$.  We wish to append a single new heap to this forest.  In order for the shrub forest to avoid 312, all labels that are used on leaves of the new heap must be larger than all existing labels, and the root can appear lower.  In particular, the labels on the new heap may be $i, 3k+2, 3k+3$ for $1 \leq i \leq 3k+1$ (and if $i \leq 3k$ then $i$ may not be the larger digit in an inversion in the original forest on $\{1,2,\dots , 3k\}$, so we know $i \neq 3k$ in this case).

As before, when we append a new 123 heap to the end of a 312-avoiding forest, we increase the number of sets of labels we can use on the next heap by 3 (since $i \in \{3k-1, 3k, 3k+1\}$ are new), and when we append a new 132 heap to the end of a 312-avoiding forest, we increase the number of sets of labels we can use on the next heap by 2 (since $i \in \{3k-1, 3k+1\}$ are new).  The number $k$ of $(1,-1)$ downsteps tells us to use the $(k+1)$st highest possible value for the root of the next appended heap.

Since both 213 and 312-avoiding forests are in bijection with the same set of lattice paths using $(1,3)$, $(2,2)$, and $(1,-1)$ steps, the two sets of forests are equinumerous.
\end{proof}

For example, the path in Figure \ref{av213} corresponds to the 213-avoiding forest \underline{7 15 14} \underline{8 9 10} \underline{11 13 12} \underline{1 5 6} \underline{2 4 3}.  Earlier, we built this forest by labeling up-segments in the path, but we can also follow the argument of Theorem \ref{T:av213eq312} for an alternate construction.  The first upstep is a $(2,2)$ step, so we begin with the forest \underline{1 3 2}.  Next, we have $d=1$ downsteps, so $i=\ell-d-1=3-1-1=1$.  Our next tree has labels 2, 3, and 4, and since the next upstep is a $(1,3)$ step, we have \underline{1 6 5} \underline{2 3 4}.  Next, we have $d=0$ downsteps, so $i=\ell-d = 4-0=4$.  Our next tree has labels 5, 6, and 7, and since the next upstep is a $(2,2)$ step, we have \underline{1 9 8} \underline{2 3 4} \underline{5 7 6}.  Next, we have $d=6$ downsteps, so $i=\ell-d-1 = 7-6-1=0$.  Our next tree has labels 1, 2, and 3, and since the next upstep is a $(1,3)$ step, we have \underline{4 12 11} \underline{5 6 7} \underline{8 10 9} \underline{1 2 3}.  Next, we have $d=2$ downsteps, so $i=\ell-d = 3-2=1$.  Our next tree has labels 2, 3, and 4, and since the next upstep is a $(2,2)$ step, we have \underline{7 15 14} \underline{8 9 10} \underline{11 13 12} \underline{1 5 6} \underline{2 4 3}.

Similarly, we can construct a 312-avoiding forest from the same path working from left to right.  The first upstep is a $(2,2)$ step, so we begin with the forest \underline{1 3 2}.  Next, we have $d=1$ downsteps, and the possible values of the next root are 1, 2, or 4.  We choose the 2nd highest value.  Our next tree has labels 2, 5, and 6, and since the next upstep is a $(1,3)$ step, we have \underline{1 4 3} \underline{2 5 6}.  Next, we have $d=0$ downsteps, and the possible values of the next root are 1, 2, 5, 6, or 7.  We choose the highest value.  Our next tree has labels 7, 8, and 9, and since the next upstep is a $(2,2)$ step, we have \underline{1 4 3} \underline{2 5 6} \underline{7 9 8}.  Next, we have $d=6$ downsteps, and the possible values of the next root are 1, 2, 5, 6, 7, 8, or 10.  We choose the 7th highest value.  Our next tree has labels 1, 11, and 12, and since the next upstep is a $(1,3)$ step, we have \underline{2 5 4} \underline{3 6 7} \underline{8 10 9} \underline{1 11 12}.  Next, we have $d=2$ downsteps, and the possible values of the next root are 1, 11, 12, or 13.  We choose the 3rd highest value.  Our next tree has labels 11, 14, and 15, and since the next upstep is a $(2,2)$ step, we have \underline{2 5 4} \underline{3 6 7} \underline{8 10 9} \underline{1 12 13} \underline{11 15 14}.

%\addtext{
By combining our results, we can enumerate forests avoiding either 213 or 312.%}

%\addtext{
\begin{theorem}%\note{Added a new theorem, giving the exact enumeration.}
\label{T:av213and312}
$\left|\mathcal{S}^2_n(213)\right|  = \left|\mathcal{S}^2_n(312)\right| = \dfrac{1}{3 m+1}\displaystyle\sum\limits_{v=0}^{m}\dbinom{3 m+1}{m-v}\dbinom{3 m+v}{v}$.
\end{theorem}

\begin{proof}
The result follows directly from Theorems~\ref{T:av213eq312} and~\ref{T:av213paths}, and Propositions~\ref{propPathBij} and~\ref{propSchroder} with $\ell=3$.
\end{proof}
%} % \addtext

The bijection presented in Theorem~\ref{T:av213eq312} can be generalized to $k$-ary shrub forests, for which we provide an outline in the interest of length. For example, the number of ternary shrubs avoiding 213 is in bijection with lattice paths using the steps $(1,-1)$, $(1,4)$ (corresponding to 1234), red $(2,3)$ (corresponding to 1243), red $(3,2)$ (corresponding to 1342), blue $(2,3)$ (corresponding to 1423), and blue $(3,2)$ (corresponding to 1432).  Notice that since we wish to use $(2,3)$ labels and $(3,2)$ labels for different kinds of heaps, we need colored labels.

For example, the path (red $(3,2)$) $(1,-1)$ (blue $(2,3)$) $(1,-1)$ $(1,-1)$ $(1,4)$ $(1,-1)^6$ corresponds to 1(11)(12)(10) 2978 3456.

In general, $k$-ary shrub forests avoiding either 213 or 312 are in bijection with lattice paths using $(1,-1)$ steps and $\frac{1}{k+1}{\binom{2k}{k}}$ steps of the form $(j,k-j+2)$ going from $(0,0)$ to $((k+2)n,0)$ and staying weakly above the $x$-axis.  To find the step corresponding to a particular $k$-ary shrub permutation $\pi$ let $i$ be the value of the smallest number that plays the role of a 2 in a 21-pattern, or $i=k+2$ if $\pi$ is the identity.  Then the corresponding step is $(k+2-(i-1),i-1)$.  This makes sense since if $i$ is the smallest number that plays the role of 2 in a 21 pattern, then all digits greater than or equal to $i$ must be larger than labels in new 213-avoiding shrubs appended to the end of the forest.  With $i-1$ smaller digits, there are $i$ sets of consecutive values that can be used on the labels of a newly appended shrub.  Using an upstep that takes $i-1$ vertical units allows $i$ choices for how many $(1,-1)$ steps may come after it.

\subsection{Avoiding 231}

%\addtext{
To enumerate $\mathcal{S}^2_n(231)$, we make use of the following celebrated result of Duchon concerning lattice paths bounded by $y=\frac{2}{3}x$.

\begin{proposition}[{Duchon~\cite[Theorem~11]{Duchon2000}}]\label{propDuchon}%\note{Added a new proposition, used in the proof of Theorem~\ref{T:av231}.}
The number of lattice paths of length $5n$ from the origin
to the line $y=\frac{2}{3}x$ with unit East and North steps that stay
weakly below the line is given by
$$\sum_{i=0}^n \frac{1}{5n+i+1}{\binom{5n+1}{n-i}\binom{5n+2i}{i}}.$$
\end{proposition}

%\note{Replaced OEIS with primary references.}
This sequence was investigated further by Banderier and Flajolet~\cite{BF2002} under a slightly different formulation, whose equivalence follows from Proposition~\ref{propPathBij}. This alternative perspective is known as ``Duchon's club model'':
\emph{People arrive at a club by pairs and leave in threesomes. What is
the number of possible scenarios from the club opening until it closes?}

We enumerate forests avoiding 231 by exhibiting a bijection with Duchon's club paths.
%} % \addtext

\begin{theorem}\label{T:av231}%\note{I've combined the two original theorems in this subsection.}
The equality $\left|\mathcal{S}^2_n(231)\right| =\displaystyle\sum\limits_{i=0}^n \dfrac{1}{5n+i+1}{\dbinom{5n+1}{n-i}\dbinom{5n+2i}{i}}$ holds.
\end{theorem}

\begin{proof}
Let $\mathcal{A}_n$ be the set of 231-avoiding permutations realizable on a binary shrub forest containing $n$ heaps.  Let $\mathcal{B}_n$ be the set of lattice paths with $3n$ East $(1,0)$ and $2n$ North $(0,1)$ steps that stay at or below the line $y=\frac{2x}{3}$.

To prove the theorem, we give a map $\phi: \mathcal{B}_n \to \mathcal{A}_n$; then we prove that it is indeed a bijection by showing that it is injective, well-defined, and surjective.  %\addtext{
The result then follows by the application of Proposition~\ref{propDuchon}.%}\note{Added use of Duchon's Proposition.}

To explain $\phi$ rigorously, we must first introduce some notation.
Given $b \in \mathcal{B}_n$, we partition $b$ into blocks of the form $E^kN$ with $k \geq 0$.  Since there are $2n$ $N$s (and $b$ ends in $N$), there are $2n$ blocks in $b$; call them $B_1, \dots, B_{2n}$.

To determine $\phi(b)$, we first construct a permutation with vertical bars between certain pairs of adjacent digits.  Given such a permutation $a$, write $a=A_1|A_2|\cdots|A_{\ell}$.  Also, given a subpermutation $A_i|A_{i+1}|\cdots|A_{i+j}$, let $A_i-A_{i+j}$ be the permutation formed by erasing all bars between these blocks. Finally, given a string of digits $a$, write $a^{i+}$ for the string formed by incrementing all digits greater than or equal to $i$ by 1.
	
We are now ready to describe $\phi$.  Let $b \in \mathcal{B}_n$ have blocks $B_1, \dots, B_{2n}$.

\begin{enumerate}
\item Let $i=2n$, and let $a=1$.
\item \begin{enumerate}
\item If block $B_{i} = N$ then let $a=1|(a)^{1+}$.
\item If block $B_{i} = E^kN$ where $k \geq 1$, then determine $m=\max(A_k)+1$.  Let \\${a=m(A_1-A_k|A_{k+1}|\cdots|A_{\ell})^{m+}}$.
\end{enumerate}
\item If $i$ is even, then let $a=1|(a)^{1+}$.
\item Decrement $i$ by 1.  \\If $i=1$ then we are done.  Return $\phi(b)=A_1-A_{\ell}$.\\If $i>1$, then return to step 2.
\end{enumerate}

%\addtext{
Recall that a value in a permutation is a \emph{left-to-right maximum} if it is larger than all the values to its left.  Notice that by construction, at the end of each step, the vertical bars appear immediately before each left-to-right maximum of $a$ (except the first one).  We add a bar in steps 2a and 3 when we introduce an ascent at the beginning of $a$ by $a=1|(a)^{1+}$ (and thus a create new left-to-right maximum).  We erase bars in step 2b when we create a new left-to-right maximum that supersedes the left-to-right maxima from a previous step.  It follows that while we do not use $B_1$ to encode additional digits of $A$, the number of $E$'s in $B_1$ is the number of left-to-right maxima of $a$ at the end of the algorithm describing the map $\phi$.  In general, notice that since $\max(A_k)=\max(A_1-A_k)=m-1$, $a=A_1|\cdots|A_{k-1}|(m-1)A'_k|(m')A'_{k+1}|\cdots$ where $A'_i$ is $A_i$ with its first entry deleted, and $m' > m-1$.

It is clear that $\phi$ is injective; consider the rightmost block where two paths differ and thus two different digits will be appended to the beginning of $a$, resulting in different outputs.  It remains to show that $\phi$ is well-defined and surjective.

To show that $\phi$ is well-defined, we must show that when we read a block $B=E^kN$ with $k \geq 0$ in step 2 (including both cases 2a and 2b), $a=A_1|A_2|\cdots|A_{\ell}$ with $\ell \geq k$.  Let $b'=B_{2n-j+1}\cdots B_{2n}$, the rightmost $j$ blocks of $b$.  We claim that the number of bars in $a$ after we have read $j$ blocks of $b$ is given by
\[
	n_N(b')-n_E(b')+\left\lfloor \frac{n_N(b')+1}{2}\right\rfloor = j-n_E(b')+\left\lfloor \frac{j+1}{2}\right\rfloor,
\]
where $n_E(W)$ (resp. $n_N(W)$) is the number of $E$ (resp. $N$) steps in word $W$.  We begin with no bars.  Every $N$ block creates one more bar in step 2a.  Every occurrence of an $E^kN$ block with $k \geq 1$ (step 2b) creates a descent and reduces the number of bars by $k-1$.  After every second block (i.e. $\left\lfloor \frac{j+1}{2}\right\rfloor$ times) we create an additional bar in step 3.  We add 1 to the number of bars to compute $\ell$.

It is immediate from the fact that $b \in \mathcal{B}_n$ lies below the line $y=\frac{2}{3}x$ that for all $b=b''b' \in \mathcal{B}_n$ we have $(3n-n_E(b'),2n-n_N(b')) \in \left\{(x,y) \mid x,y \in \mathbb{N} \text{ and } 2x \geq 3y\right\}$ so $2(3n-n_E(b')) \geq 3(2n-n_N(b')) \iff n_E(b') \leq 3/2 n_N(b')$.  Now let $b' = B_{2n-j+1}\cdots B_{2n}$ as above, let $b^* = B_{2n-j}B_{2n-j+1}\cdots B_{2n}$, and suppose $B_{2n-j}=E^kN$.  We must show $k \leq \ell$.  $k=n_E(b^*)-n_E(b') \leq \frac{3}{2}n_N(b^*)-n_E(b') = \frac{3}{2}(j+1)-n_E(b') \leq (j+1) - n_E(b') + \frac{j+1}{2}$, as desired. Therefore, step 2 above is always possible.

Now, we check that $\phi(b) \in \mathcal{S}^2_n(231)$.  First, $a$ is indeed a permutation at each step in the algorithm for $\phi$.  We begin with $a=1$.  Each time we use step 2 or step 3 to prepend a new digit, all larger digits are incremented by 1 so $a$ always consists of a string of consecutive non-repeating digits.  Next, $\phi(b)$ has length ${3n}$ since we begin with 1 digit and obtain $(2n-1)$ new digits from step 2 and $n$ new digits from step 3.  Finally, $\phi(b)$ avoids 231 since the digit inserted into the first position of $a$ must play the role of 2 in the 231 pattern.  It is impossible for this digit to play the role of 2 in steps 2a and 3 since the first digit is the smallest.  In case 2b, the permutation has all digits smaller than $m$ appearing before all digits larger than $m$.  Therefore, $\phi: \mathcal{B}_n \to \mathcal{A}_n$ is well-defined.

%Now, since these bars come from $b \in \mathcal{B}_n$, we know that at any point reading from right to left the number of $E$'s is at most $\frac{3}{2}$ times the number of $N$'s.
%Suppose to the contrary that
%$$n_E(B_{\ell-1}) > 1+ (2n-\ell+1) - n_E(B_{\ell} \cdots B_{2n}) +  \left\lfloor \frac{(2n-\ell+1)+1}{2}\right\rfloor .$$  Then
%$$n_E(B_{\ell-1}\cdots B_{2n}) > 1+ (2n-\ell+1)  +  \left\lfloor \frac{(2n-\ell+1)+1}{2}\right\rfloor .$$  For simplicity, let $2n-\ell+2 = m$, so
%$$n_E(B_{\ell-1}\cdots B_{2n}) > m  +  \left\lfloor \frac{m}{2}\right\rfloor .$$  Notice that $m$ is the number of blocks we have seen so far, so it is the number of $N$s in $B_{\ell-1} \cdots B_{2n}$.  If $m$ is odd, then $m=2z+1$ for some integer $z$, so
%$$n_E(B_{\ell-1}\cdots B_{2n}) > 2z+1  +  \left\lfloor \frac{2z+1}{2}\right\rfloor  =2z+1+z  = 3z+1 = 3\frac{m-1}{2}+1 = \frac{3}{2}m - \frac{1}{2}.$$
%Since we are hoping to show $n_E(B_{\ell-1}\cdots B_{2n}) > \frac{3}{2}m$ for a contradiction, we appear to be at an impasse.  However, notice that $n_E(B_{\ell-1}\cdots B_{2n})$ is an integer while $\frac{3}{2}m$ is a fraction when $m$ is odd.  In this case, an integer larger than $\frac{3}{2}m - \frac{1}{2}$ must also be larger than $\frac{3}{2}m$.
%\\On the other hand, if $m$ is even, then $m=2z$ for some integer $z$, so
%$$n_E(B_{\ell-1}\cdots B_{2n}) > 2z +  \left\lfloor \frac{2z}{2}\right\rfloor = 2z+z = 3z = \frac{3}{2}m.$$
%In both cases, the number of $E$s in $B_{\ell-1}\cdots B_{2n}$ exceeds $\frac{3}{2}$ the number of $N$'s, which contradicts $b \in \mathcal{B}_n$.

Finally, to show that $\phi$ is surjective, we show that every $a \in \mathcal{A}_n$ has a path $b \in \mathcal{B}_n$ that is mapped to it.  Given $a$, we can certainly reverse the encoding of ascents and descents prescribed by $\phi$.  To be rigorous, given $a \in \mathcal{A}_n$, we build the corresponding path $b$ in the following way:

\begin{enumerate}
\item Let $b$ be the empty path and let $i=3n-1$.
\item
\begin{enumerate}
\item If $i=3z+1$ for some integer $z$, then $b$ remains unchanged. \\(This is the inverse of step 3 above.)
\item If $i\neq 3z+1$ for some integer $z$ and $a_{i}<a_{i+1}$ then $b=Nb$. \\(This is the inverse of step 2a above.)
\item If $i\neq 3z+1$ for some integer $z$ and $a_{i}>a_{i+1}$, then $b=E^kNb$ where $k$ is the number of left-to-right maxima of $a_{i+1}\cdots a_{3n}$ that are less than $a_{i}$. \\(This is the inverse of step 2b above).
\end{enumerate}
\item Decrement $i$.
\\ If $i=1$, then $\phi^{-1}(a)=E^{3n-n_E(b)}Nb$.
\\If $i>1$ then return to step 2.
\end{enumerate}

To show $\phi$ is surjective, we must verify that any path $b$ formed in this way stays below the line $y=\frac{2}{3}x$.  Suppose we have $a \in \mathcal{A}_n$ such that $\phi^{-1}(a) \notin \mathcal{B}_n$.  Consider the first block $B_{\ell-1}=E^kN$ (reading from right to left) where the corresponding path goes above the line $y=\frac{2}{3}x$, but $B_{\ell}\cdots B_{2n}$ is below the line.
We have $n_E(B_{\ell}\cdots B_{2n}) \leq \frac{3}{2}n_N(B_{\ell}\cdots B_{2n})$ but $n_E(B_{\ell-1}\cdots B_{2n}) > \frac{3}{2}n_N(B_{\ell-1}\cdots B_{2n})$, or equivalently, $n_E(B_{\ell}\cdots B_{2n})+k > \frac{3}{2}\left(n_N(B_{\ell}\cdots B_{2n})+1\right)$.
Since $b$ comes from the map $\phi^{-1}$, we know that $$k \leq 1+n_N(B_{\ell}\cdots B_{2n}) - n_E(B_{\ell}\cdots B_{2n})+\left\lfloor \frac{n_N(B_{\ell}\cdots B_{2n})+1}{2}\right\rfloor.$$  Therefore,
$$n_E(B_{\ell}\cdots B_{2n}) + k \leq 1+n_N(B_{\ell}\cdots B_{2n}) +\left\lfloor \frac{n_N(B_{\ell}\cdots B_{2n})+1}{2}\right\rfloor.$$ But by assumption $n_E(B_{\ell}\cdots B_{2n}) + k >\frac{3}{2}\left(n_N(B_{\ell}\cdots B_{2n})+1\right),$ so
$$\frac{3}{2}n_N(B_{\ell}\cdots B_{2n})+\frac{3}{2} < 1 + n_N(B_{\ell}\cdots B_{2n}) +\left\lfloor \frac{n_N(B_{\ell}\cdots B_{2n})+1}{2}\right\rfloor.$$
Simplifying, $$\frac{1}{2}n_N(B_{\ell}\cdots B_{2n}) + \frac{1}{2} < \left\lfloor\frac{n_N(B_{\ell}\cdots B_{2n})+1}{2}\right\rfloor.$$
If $n_N(B_{\ell}\cdots B_{2n})$ is even, then $n_N(B_{\ell}\cdots B_{2n}) = 2z$ for some integer $z$, so $z+\frac{1}{2}<\left\lfloor \frac{2z+1}{2}\right\rfloor = z$, which is a contradiction.  On the other hand, if $n_N(B_{\ell}\cdots B_{2n})$ is odd, then $n_N(B_{\ell}\cdots B_{2n}) = 2z+1$ for some integer $z$, so $z+1<\left\lfloor \frac{2z+2}{2}\right\rfloor = z+1$, which is also a contradiction.  Therefore, any permutation in $\mathcal{A}_n$ does indeed correspond to a path below the line $y=\frac{2}{3}x$.
\end{proof}

For example, consider $b=EENENEEEEEENNENNENEN$.  We have $n=4$.  Here $B_1=EEN$, $B_2=EN$, $B_3=EEEEEEN$, $B_4=N$, $B_5=EN$, $B_6=N$, $B_7=EN$, and $B_8=EN$.
\begin{itemize}
\item We begin with $i=8$ and $a=1$.
\item Since $B_8=EN$ and $k=1$, we see that $m=\max(1)+1=2$, so we have $a=21$ by step 2b.

Since $i=8$ is even, $a=1|32$ by step 3.

Decrement $i$ to 7 and return to step 2.

Since $B_7=EN$ and $k=1$, we see that $m=\max(1)+1=2$, so we have $a=21|43$ by step 2b.

Decrement $i$ to 6 and return to step 2.

Since $B_6=N$, $a=1|32|54$ by step 2a.

Since $i=6$ is even, $a=1|2|43|65$ by step 3.

Decrement $i$ to 5 and return to step 2.

Since $B_5=EN$ and $k=1$, we see that $m=\max(1)+1=2$, so we have $a=21|3|54|76$ by step 2b.

Decrement $i$ to 4 and return to step 2.

Since $B_4=N$, $a=1|32|4|65|87$ by step 2a.

Since $i=4$ is even, $a=1|2|43|5|76|98$ by step 3.

Decrement $i$ to 3 and return to step 2.

Since $B_3=EEEEEEN$ and $k=6$, we see that $m=\max(98)+1=10$, so we have $a=(10)124357698$ by step 2b.

Decrement $i$ to 2 and return to step 2.

Since $B_2=EN$ and $k=1$, we see that $m=\max((10)124357698)+1=11$, so we have $a=(11)(10)124357698$ by step 2b.

Since $i=2$ is even, $a=1|(12)(11)2354687(10)9$ by step 3.

\item Decrement $i$ to 1.  Since $i=1$, we are done.  \\
$\phi(EENENEEEEEENNENNENEN) = 1(12)(11)2354687(10)9 \in \mathcal{S}^2_4(231)$.
\end{itemize}

%\begin{theorem}\label{T:av231f}%\note{I've combined the two original theorems in this subsection.}
Notice that Theorem~\ref{T:av231} generalizes as follows:  $\mathcal{F}^k_n(231)$ is equinumerous with paths with unit East and North steps from $(0,0)$ to $((k+1)n, kn)$ bounded above by $y=\frac{k}{k+1}x$.
%\end{theorem}
Given a path bounded by $y=\frac{k}{k+1}x$, we build the permutation for a 231-avoiding $k$-ary shrub forest from right to left.  Each $E^iN$ block encodes information about a leaf of the forest, and after every set of $k$ leaves we add a new right-to-left minimum as a root.

\subsection{Avoiding 321}

%\addtext{
In contrast to the situation with other patterns of length 3, we are unable to present an explicit expression for the enumeration of forests avoiding 321.
Nor do we exhibit a bijection with a family of lattice walks (though it is possible that such a bijection exists).
However, using the techniques of analytic combinatorics, we are able to determine the generating function for this set, and its growth rate.
%} % \addtext

\subsubsection{Deriving a functional equation}

%\addtext{
We begin by deriving a functional equation for a bivariate generating function, making use of a catalytic variable.%}

Recall that an \emph{inversion} is an occurrence of 21 in a permutation.
Given a permutation $\pi$, let its \emph{last inversion foot} be the lower point of the rightmost inversion of $\pi$ (if there is one), and let the statistic $\lif(\pi)$ count the number of entries of $\pi$ with value greater than that of its last inversion foot. In other words, if $|\pi|=n$ and $\pi(i)$ is the rightmost entry that is not a left-to-right maximum, then $\lif(\pi) = n-\pi(i)$.  If $\pi$ is an increasing permutation (i.e., $\pi =12\cdots n$), then we set $\lif(\pi) = n$.

\begin{theorem}\label{thmFunctEqn}
  Let $H(x,u)$ be the bivariate generating function for $321$-avoiding binary shrub forests, where the coefficient of $x^nu^k$ is the number of forests consisting of $n$ shrubs whose underlying permutation $\sigma$ (of length $3n$) satisfies $\lif(\sigma)=k$.

  Then, $H(x,u)$ satisfies the functional equation
\[
		H(x,u) \;=\; 1 \:+\:
          \frac{xu}{1-u}\Bigg(                     \frac{u^2(1-3u+2u^2-u^3)}{(1-u)^2} H(x,u)
                                              \:+\: \frac{1-3u+3u^2-u^3+u^4}{(1-u)^2}  H(x,1)\]\[\qquad\qquad
				                \:+\: \frac{2-2u-u^2}{1-u}     \frac{\partial H}{\partial u}(x,1)
                                             \:+\: \frac{1}{2}           \frac{\partial^2H}{\partial u^2}(x,1)\Bigg) .
\]
% H(x,u) = 1+x*u/(1-u)*(u^2*(-u^3+2*u^2-3*u+1)/(1-u)^2*H(x,u)+(u^4-u^3+3*u^2-3*u+1)/(1-u)^2*eval(H(x,u),u=1)+(-u^2-2*u+2)/(1-u)*eval(diff(H(x,u),u),{u = 1})+1/2*eval(diff(diff(H(x,u),u),u),{u = 1}))
\end{theorem}
\begin{proof}
Consider $\mathcal{S}^2_n(321)$. These are permutations avoiding $321$ of length $3n$ such that each of the $n$ blocks of three consecutive values is an occurrence of either a $123$ or a $132$.

Let us consider how a permutation in $\mathcal{S}^2_n(321)$ may be extended by adding three points to its right in such a way that the resulting permutation is also a member of $\mathcal{S}^2_n(321)$.
First of all, observe that whenever a point is added, it must have value greater than that of the last inversion foot, or else a $321$ will be formed.

Suppose $\lif(\sigma)=k$.
Adding a new maximum entry to $\sigma$ in the rightmost position does not alter the last inversion foot. In this case, the $\lif$ statistic of the resulting permutation is $k+1$.
On the other hand, if a new point is added in any other valid location it
results in a permutation whose $\lif$ statistic is at least 1 but is not greater than $k$, as $321$ avoidance implies that the new point must have value greater than that of the previous last inversion foot.

Define the following four transition rules for the $\lif$ statistic:
\begin{description}
  \item[$\qquad$\textmd{\textsf{A}}:] $\quad k \to \{1,\ldots,k\}$,
  \item[$\qquad$\textmd{\textsf{B}}:] $\quad k \to \{1,\ldots,k+1\}$,
  \item[$\qquad$\textmd{\textsf{C}}:] $\quad k \to \{k+1\}$,
  \item[$\qquad$\textmd{\textsf{D}}:] $\quad k \to \{1,k+1\}$.
\end{description}
We now verify that the only valid ways to extend a permutation in $\mathcal{S}^2_n(321)$ (i.e., the only ways to add a $123$ or $132$ to the right without creating a $321$) are either \textsf{A} then \textsf{B} then \textsf{B}, or else \textsf{C} then \textsf{C} then \textsf{D}.

To see this, first observe that if we wish to add a $123$ or $132$ pattern to the right of $\sigma$ using values $\{|\sigma|+1, |\sigma|+2, |\sigma|+3\}$, then adding the first two entries ($12$ if adding a $123$, $13$ if adding a $132$) corresponds to applying $\textsf{C}$ twice. Appending the rightmost entry effects the transition $k \to \{k+1\}$ if a $123$ is created and effects the transition $k \to \{1\}$ if a $132$ is created. Therefore, if the $123$ or $132$ pattern is added entirely above the maximum entry of $\sigma$, the rule $\textsf{C}$ then $\textsf{C}$ then $\textsf{D}$ captures all possible transitions for the $\lif$ statistic.

Suppose next that the first two entries in the pattern to be appended to $\sigma$ are added below the maximum entry of $\sigma$. This corresponds to applying rule $\textsf{A}$ twice. Moreover, we are now forbidden from creating a $132$ pattern, as the maximum of $\sigma$ together with the $32$ would be a $321$ pattern. The only option is to create a $123$ pattern, and the placement of the last entry corresponds to rule $\textsf{B}$.

If on the other hand the first entry to be added lies below the maximum of $\sigma$ but the second entry lies above (corresponding to applying rule $\textsf{A}$ and then rule $\textsf{C}$), either a $123$ or a $132$ pattern can be created and the insertion of the last entry corresponds to rule $\textsf{B}$. This case ($\textsf{A}$ then $\textsf{C}$ then $\textsf{B}$), together with the previous case ($\textsf{A}$ then $\textsf{A}$ then $\textsf{B}$) combine to give the transition rule $\textsf{A}$ then $\textsf{B}$ then $\textsf{B}$.

Each of these transition rules can be represented as a \emph{linear operator} acting on the generating functions, as follows (see~\cite[Exercises~III.22, p.199 and V.20, p.365]{FS2009}):

\begin{center}
  \begin{tabular}{ll}
    $\oper_\ssA\big[u^k\big] = u+u^2+\ldots+u^k = \tfrac{u}{1-u}(1-u^k)$ &
    $\oper_\ssA\big[f(u)\big] = \tfrac{u}{1-u}\big(f(1)-f(u)\big)$ \\[6pt]
    $\oper_\ssB\big[u^k\big] = u+u^2+\ldots+u^{k+1} = \tfrac{u}{1-u}(1-u^{k+1})$ \qquad\quad &
    $\oper_\ssB\big[f(u)\big] = \tfrac{u}{1-u}\big(f(1)-u f(u)\big)$ \\[6pt]
    $\oper_\ssC\big[u^k\big] = u^{k+1}$ &
    $\oper_\ssC\big[f(u)\big] = u f(u)$ \\[6pt]
    $\oper_\ssD\big[u^k\big] = u+u^{k+1}$ &
    $\oper_\ssD\big[f(u)\big] = u\big(f(1)+f(u)\big)$
  \end{tabular}
\end{center}

Thus $H(x,u)$ satisfies the functional equation
\[
H(x,u) \;=\; 1 \:+\: x\Big( \oper_\ssB\big[\oper_\ssB\big[\oper_\ssA\big[H(x,u)\big]\big]\big] + \oper_\ssD\big[\oper_\ssC\big[\oper_\ssC\big[H(x,u)\big]\big]\big] \Big) ,
\]
where the initial $1$ corresponds to the empty shrub forest.

Observing by arithmetic of formal power series that if
\[
 f(u) = \frac{H(x,1)-H(x,u)}{1-u} , \quad \text{then} \quad f(1) = \frac{\partial H}{\partial u}(x,1),
\]
a computer algebra system can easily simplify the resulting expression to give the functional equation in the statement of the theorem.
\end{proof}

The most efficient way to generate coefficients of $H(x,u)$, and hence of $H(x,1)$, is to iterate the transition rules.
Greater performance is achieved by combining the valid sequences of three steps into one rule,
\begin{description}
  \item[$\qquad$\textmd{\textsf{ABB}\raisebox{1pt}{+}\textsf{CCD}}:] $\quad k \:\to\: \Big\{ 1^{\binom{k+2}{2}},\, 2^{\binom{k+2}{2}-1},\, 3^{\binom{k+1}{2}},\, 4^{\binom{k}{2}},\, \ldots,\, k^6,\, (k+1)^3,\, k+2,\, k+3 \Big\}$,
\end{description}
the right hand side being a multiset in which multiplicities are represented by exponents.

Using \emph{Mathematica}~\cite{Mathematica}, we were able to calculate 993 values in the enumeration sequence for binary shrub forests avoiding $321$. The first ten terms (including the empty forest) are:
\[
   1,\, 2,\, 37,\, 866,\, 23285,\, 679606,\, 20931998,\, 669688835,\, 22040134327,\, 741386199872 .
\]
See A257995 in~\cite{OEIS} for more.

A \emph{Maple} program was then used to find a possible polynomial equation satisfied
by the generating function. The first $250$ terms sufficed to suggest that $H(x,1)$ was a root of the polynomial given in the statement of Theorem~\ref{thmMinPoly} below, and hence is algebraic.

% ----------------------------------------------------------------
\subsubsection{Confirming algebraicity}

To prove that $H(x,u)$, and hence $H(x,1)$, is algebraic, we make use of a general result of Bousquet-M\'elou and Jehanne~\cite{BMJ2006}. To state their theorem, we first need to introduce some notation.
Suppose
\[
   F(x,u) \;=\; f_0(x) \:+\: f_1(x)u \:+\: f_2(x)u^2 \:+\: \ldots
\]
is a bivariate formal power series. We define a sequence of operators $\delop,\delop^2,\delop^3,\ldots$, that discard the low order terms, as follows:
\[
   \delop^i \big[F(x,u)\big] \;=\; f_i(x) \:+\: f_{i \:+\: 1}(x)u \:+\: f_{i \:+\: 2}(x)u^2 \:+\: \ldots .
\]

\begin{theorem}[Bousquet-M\'elou and Jehanne~\cite{BMJ2006}, Theorem 3]\label{thmBMJAnalytic}
  If $P$ and $Q$ are polynomials, then the functional equation
  \[
    F(x,u) \;=\; P(u) \:+\: x
    Q\big(F(x,u),
           \delop   \big[F(x,u)\big],
           \delop^2 \big[F(x,u)\big],
           \ldots,
           \delop^k \big[F(x,u)\big],
           x,
           u
    \big)
  \]
  has a unique solution for $F(x,u)$ that is a formal power series in $x$ whose coefficients are polynomials in $u$. Moreover, this solution is algebraic.
\end{theorem}

The functional equation in Theorem~\ref{thmFunctEqn} is not in the appropriate form to apply Theorem~\ref{thmBMJAnalytic} directly. Setting $G(x,u) = H(x,u+1)$, we have
\[
	\frac{\partial H}{\partial u}(x,1) = \frac{\partial G}{\partial u}(x,0) \quad \text{ and } \quad \frac{\partial^2 H}{\partial u^2}(x,1) = \frac{\partial^2 G}{\partial u^2}(x,0)
\]
Upon making these substitutions, the functional equation stated in Theorem~\ref{thmFunctEqn} becomes
\[
	G(x,u) = 1 + \frac{x(1+u)}{u}\left(\frac{1 + 4u + 6u^2 + 5u^3 + 3u^4 + u^5}{u^2}G(x,u)\right. \qquad\qquad\qquad\qquad
\]
\[
	\qquad\qquad\qquad\qquad\qquad \left.  - \frac{1+4u+6u^2 + 3u^3 + u^4}{u^2} G(x,0) - \frac{1+4u+u^2}{u} \frac{\partial G}{\partial u}(x,0) - \frac{1}{2} \frac{\partial^2 G}{\partial u^2}(x,0)\right).
\]
In addition, the definition of the $\delop$ operator implies that
\[
	\frac{\partial^k G}{\partial u^k}(x,0) = k!(\delop^k[G(x,u)] - u\delop^{k+1}[G(x,u)]),
\]
for $k \geq 0$. Hence, the functional equation for $G$ can be transformed into the form in Theorem~\ref{thmBMJAnalytic} by making this substitution for $k=0,1,2$, clearing denominators, and rearranging.
%
%
%However, if we let $G(x,v)=H(x,v+1)$, and observe by arithmetic of formal power series and the definition of $\delop$ that
%\[
%	\frac{\partial^k G}{\partial v^k}(x,1) = k!\big(\delop^k \big[G(x,v)\big] \:-\: v \delop^{k+1} \big[G(x,v)\big]\big),
%\]
%then upon applying this substitution and simplifying, our equation becomes\begin{multline*}
%   G(x,v) \;=\; 1 \:+\: x(1 + v) \Big(       (2 + 2 v + v^2)            G(x,v)
%                                       \:+\: (5 + 3 v + v^2) x \delop   \big[G(x,v)\big] \\
%                                       \:+\: (4 + v)           \delop^2 \big[G(x,v)\big]
%                                       \:+\:                   \delop^3 \big[G(x,v)\big]
%                                 \Big) ,
%\end{multline*}
%which has the required form.
Thus $G(x,u)$ is algebraic, and hence so are $H(x,u)$ and $H(x,1)$ and its derivatives.

% ----------------------------------------------------------------
\subsubsection{Solving the functional equation}

We would like to solve our functional equation to get an explicit expression for $H(x,1)$.
Unfortunately, this is not possible. However, it is possible to determine the minimal polynomial of which $H(x,1)$ is a root.

\begin{theorem}\label{thmMinPoly}
  The generating function $H_0(x)=H(x,1)$ for $321$-avoiding binary shrub forests is a root of the polynomial
	\begin{align*}
		&             (4x^2+4x+1)
                \:-\: (x^4-24x^3+8x^2-54x+1)             H_0(x)^2    \\
        & \quad \:+\: (15x^4+24x^3-71x^2-54x)            H_0(x)^4
		        \:+\: (18x^5-215x^4+2x^3-360x^2)         H_0(x)^6    \\
        & \quad \:+\: (3x^6-228x^5-213x^4+162x^3+729x^2) H_0(x)^8
		        \:-\: (138x^6-354x^5-1053x^4)            H_0(x)^{10} \\
        & \quad \:-\: (36x^7-751x^6-486x^5)              H_0(x)^{12}
		        \:-\: (3x^8-420x^7-54x^6)                H_0(x)^{14} \\
        & \quad \:+\: 123x^8                             H_0(x)^{16}
                \:+\: 18x^9                              H_0(x)^{18}
                \:+\: x^{10}                             H_0(x)^{20} .
	\end{align*}
  Consequently, the growth rate of $321$-avoiding binary shrub forests is approximately 39.88873, the greatest real root of the quartic polynomial
  \[
     729x^4-28674x^3-15505x^2-25758x+621 .
  \]
\end{theorem}

Given a suitable functional equation,
Bousquet-M\'elou and Jehanne~\cite{BMJ2006} present a way of setting up a system of polynomial equations that can then be solved to yield a polynomial having a root that is the desired generating function.
They suggest that the ``laziest approach'' is to feed this system of equations to a Gr\"obner basis package and let it work.
Unfortunately, as they comment, ``this lazy approach often fails, because the computation tends to take forever''.
This has also been our experience. Submission of the appropriate equations to \emph{Singular}, a computer algebra system optimized for working with polynomials, yielded no output after a week of processing.

We use a more practical strategy, derived from the results in~\cite{BMJ2006}
(see also Section~4 of~\cite{FS2001}). For the necessary algebraic manipulation, \emph{Maple}~\cite{Maple} %\note{Added citation for \emph{Maple}.}
 was used.

\begin{proof}[Proof of Theorem~\ref{thmMinPoly}]

To start with, since our functional equation is linear in $H(x,u)$,
it can be expressed in the form
\[
K(x,u) H(x,u) \;=\; P(H_0(x),H_1(x),H_2(x),x,u),
\]
where both $K$ and $P$ are polynomials that we omit for brevity and
$H_i(x) = \displaystyle\frac{\partial^i H}{\partial u^i} (x,1)$.

Thus,
we can use the kernel method to eliminate both $H(x,u)$ and $u$.
Observe that the kernel $K(x,u)=0$ if and only if $P(H_0(x),H_1(x),H_2(x),x,u)=0$.
We want to eliminate $u$ from this pair of equations.

To do so, we calculate the \emph{resultant} of $K$ and $P$.
The resultant of two polynomials is a monomial multiple of a polynomial in their coefficients which has the property that it is equal to zero if and only if the polynomials have a common root. The resultant is given by the determinant of a matrix (known as the Sylvester matrix) whose entries are coefficients of the polynomials.

Let $R$ be the resultant of $K$ and $P$ with respect to $u$. Then we have
\[
R(H_0(x),H_1(x),H_2(x),x) \;\equiv\; 32x^5 R_1(H_0(x),H_1(x),H_2(x),x) \;=\; 0,
\]
where $R_1$ is a large polynomial that cannot be factored.
%
%Given an appropriate functional equation of the form
%\[
%  P(F_0(x),F_1(x),\ldots,F_k(x),x) \;=\; 0,
%\]
%where $P$ is a
%polynomial, it can be shown  that any combinatorial solution for one of the $F_i(x)$ is a multiple root of $P$.
%

To eliminate $H_1(x)$ and $H_2(x)$, the \emph{discriminant} of $P$ can be used.
The discriminant of a polynomial is a polynomial function of its coefficients which has the property that it is equal to zero if and only if the original polynomial has a multiple root. For example, it is well known that the discriminant of the quadratic $ax^2+bx+c$ with respect to $x$ is $b^2-4ac$.\footnote{Resultants and discriminants are closely related, the discriminant of $P$ with respect to $x$ being, up to a monomial factor, the resultant of $P(x)$ and $P'(x)$.}

We now apply this approach twice. We do not give the polynomials involved explicitly, as they would cover many pages.
Firstly, taking the discriminant of $R_1$ with respect to $H_1(x)$ yields a new equation
\[
		S(H_0(x), H_2(x), x) \;\equiv\; S_1(x)^2 S_2(H_0(x), H_2(x), x)^2 \;=\; 0,
\]
where $S_1$ is a polynomial only in $x$, and $S_2$ cannot be factored further.

Then, taking the discriminant of $S_2(H_0(x), H_2(x), x)$ with respect to $H_2(x)$ yields
\[
	T(H_0(x), x) \;\equiv\; T_1(x)^2 T_2(H_0(x), x)^2 T_3(H_0(x),x)^6 \;=\; 0.
\]
where $T_1$, $T_2$ and $T_3$ are polynomials.
Thus both $T_2$ and $T_3$ are
possibilities for the minimal polynomial of $H(x,1)=H_0(x)$.

We rule out the first choice by observing that $T_2(1,0)$ is nonzero, whereas $T_3(1,0) = 0$
as is required from considering the constant term of the series expansion of $H_0(x)$.
$T_3$ is, in fact, the polynomial presented in the statement of the theorem, which is the same as the minimal polynomial that we were able to guess based on the first $250$ terms of the enumeration sequence.
	
The growth rate of the $321$-avoiding binary shrub forests is then determined from the minimal polynomial by taking its discriminant with respect to $H_0(x)$. The growth rate is given by the reciprocal of one of the positive real roots of the discriminant (see~\cite[Note VII.36, p.504]{FS2009}). As the growth rate must be at least $1$, the positive real root whose reciprocal gives the growth rate must be at most $1$. This leaves only one candidate, which is the greatest real root of the quartic polynomial in the statement of the theorem.
\end{proof}

\section{Summary}\label{sectSummary}

Throughout this paper we have studied forests of binary shrubs that avoid any permutation pattern of length 3.  This adds a new restriction to the classical pattern avoidance problem by requiring the digits $\pi_{3i+1}\pi_{3i+2}\pi_{3i+3}$ to form a 123 or a 132 pattern for all $i$.

Remarkably, forests avoiding a single pattern $\rho \in \{123, 132, 213, 231, 312\}$ are in bijection with lattice paths in a wedge.
%\addtext{
It would be interesting to explore whether this phenomenon is more widespread.
Are other similar pattern-avoiding structures equinumerous to such lattice paths? If so, is it possible to develop a more general theory to explain this?
%}

In contrast to the other patterns, the enumeration of forests avoiding 321 required us to use the machinery of analytic combinatorics. %\addtext{
However, perhaps, in this case too, there is a connection to lattice paths that remains to be uncovered.%}

One natural generalization also merits further investigation.  In Theorem~\ref{T:av132} and the discussion after Theorems~\ref{T:av213and312} and ~\ref{T:av231}, we generalized our results to the case of $k$-ary shrubs rather than binary shrubs.  More of our results could be generalized in this way or to forests of tree structures other than shrubs.

\section*{Acknowledgments}

We grateful to the University of Wisconsin - Eau Claire (UWEC) Department of Mathematics and Office of Research and Sponsored Programs for supporting work done by four of the coauthors at UWEC during the 2014-2015 academic year.  The authors also thank Alex Burstein for organizing the Special Session on Patterns in Permutations and Words at the spring 2015 Eastern Sectional Meeting of the American Mathematical Society at Georgetown University, which allowed collaboration on a then-open case in this manuscript.

\end{document}